\documentclass[12pt,leqno]{amsart}
\usepackage{amsmath,amsthm,amsfonts,amssymb, amscd,amsbsy,multirow}
\usepackage{hyperref}
\usepackage{graphicx}
\usepackage{xcolor}
\graphicspath{ }
\usepackage{wrapfig}
\usepackage{accents}
\usepackage{caption}
\usepackage{subcaption}
\usepackage{calligra}

\numberwithin{figure}{section}

\theoremstyle{plain}
\newtheorem{thm}{Theorem}[section]
\newtheorem{lem}[thm]{Lemma}

\theoremstyle{definition}

\theoremstyle{remark}

\usepackage{mathtools}
\title[On Ricci Solitons]{Triviality Results and Conjugate Radius Estimation of Ricci Solitons}
\author[A. A. Shaikh, P. Mandal and V. A. Babu]{Absos Ali Shaikh$^{*1}$, Prosenjit Mandal$^{2}$ and V. Amarendra Babu$^{3}$}
\address{$^{1,2}$Department of Mathematics,\newline The University of Burdwan, Golapbag,\newline Burdwan-713104,\newline West Bengal, India.}

\address{$^{3}$Department of Mathematics,\newline Acharya Nagarjuna University, Nagarjuna Nagar,\newline Guntur, Andhra Pradesh, India.}
\email{$^1$aask2003@yahoo.co.in, aashaikh@math.buruniv.ac.in}
\email{$^2$prosenjitmandal235@gmail.com}
\email{$^3$amarendrab4@gmail.com}
\begin{document}
\begin{abstract}
The investigation of Ricci solitons is the focus of this work. We have proved triviality results for compact gradient Ricci soliton under certain restriction. Later, a rigidity result is derived for a compact gradient shrinking Ricci soliton. Also, we have estimated the conjugate radius for non-compact gradient shrinking Ricci solitons with superharmonic potential. Moreover, an upper bound for the conjugate radius of Ricci soliton with concircular potential vector field is determined. Finally, it is proved that a non-compact gradient Ricci soliton with a pole and non-negative Ricci curvature is non-shrinking.
\end{abstract}
\noindent\footnotetext{
$^*$ Corresponding author.\\
$\mathbf{2020}$\hspace{5pt}Mathematics\; Subject\; Classification: 53C20; 53C25; 53E20.} 
\maketitle

\section{Introduction and Results}
In the mid 1980's Hamilton \cite{hamilton1, hamilton2} introduced the concept of Ricci flow by the governing equation: 
$$\frac{\partial g}{\partial t} =-2Ric.$$ 
The study of Ricci solitons is a crucial part in the treatment of Ricci flow. Ricci solitons, as self similar solutions to Hamilton's Ricci flow, are natural generalizations of Einstein  metrics. They play an important role in the singularity analysis of the Ricci flow (see \cite{hamilton} for more details). In the past decade, a significant number of results have been obtained in classifying or understanding the geometry of Ricci solitons.
For a study of Ricci flow and Ricci solitons, we refer the reader to \cite{Cao2006, CK04, Chow-B, cunha2023, CA20, Absos23} and the references therein.

An $n$-dimensional Riemannian manifold $(M, g)$ with Ricci tensor $Ric$ is called a Ricci soliton if the Lie derivative  $\mathcal{L}_Xg$ along a smooth vector field $X$ satisfies
 \begin{equation}\label{Ricci soliton}
\frac{1}{2}\mathcal{L}_Xg + {\rm Ric} = \lambda g,
 \end{equation}
where $\lambda$ is a real constant.

The Ricci soliton is referred to as a gradient Ricci soliton if $X$ is the gradient of a smooth function $f$ on $M$. Thus, (\ref{Ricci soliton}) implies
\begin{equation}\label{Grad Ricci soliton}
  \nabla^2f + Ric = \lambda g,
 \end{equation}
where $\nabla^2 f$ represents the Hessian of $f$. Also, for $\lambda < 0$ (resp.,  $\lambda = 0$ or $\lambda > 0$), a Ricci soliton is called expanding (resp., steady or shrinking). Moreover, when the vector field $X$ is Killing or $f$ is a constant, then the Ricci soliton is called trivial.
Throughout the paper we have assumed $M$ as a complete, connected and orientable Riemannian manifold without boundary.

For any fixed $x\in M$ and $v\in T_x M$, there exists a unique geodesic $\sigma_v$ such that $\sigma_v (0)=x$ and $\sigma'_v(0)=v$, then the exponential map is defined as follows:
\begin{equation*}
exp_x:T_x M \rightarrow M, \text{ $exp_x(v)$}=\sigma_v(1), v\in T_x M.
\end{equation*}
Now, we recall the definition of the injectivity radius $r_{inj}(M)$ and the conjugate radius $r_{conj}(M)$ of $M$ which are defined respectively as follows (see, \cite{Zhu2022}):
\begin{equation*}
r_{inj}(M)=\inf_{x\in M}\sup_{r}\{r \text{ }| \text{ }exp:B(x,r) \rightarrow exp (B(x,r)) \text{ is a diffeomorphism} \},
\end{equation*}
and
\begin{equation*}
r_{conj}(M)=\inf_{x\in M}\sup_{r}\{r \text{ }| \text{ }exp:B(x,r) \rightarrow exp (B(x,r)), \text{ $B(x,r)$ contains no critical points of $exp$}\},
\end{equation*}
where $B(x,r)$ represents the open ball with center at $x$ and radius $r>0$.

It is well-known that $M$ is said to be a manifold with a pole at $p$ if $exp_p:T_p M \rightarrow M$ is a diffeomorphism. 
Also, we have $r_{inj}(M)\leq r_{conj}(M),$ based on the definition of injectivity radius and conjugate radius of $M$.

Zhu \cite{Zhu2022} studied conjugate and injectivity radius on complete, non-compact Riemannian manifolds and obtained an integral condition for the scalar curvature along with the determination of an upper bound of the conjugate radius. In \cite{Green}, Green proved the maximal conjugate radius theorem, which states that if $M$ is a compact Riemannian manifold without boundary, and its scalar curvature $R$ is bounded below by $ n(n-1)$, then $r_{conj}(M)\leq \pi$, and further, $M$ is isometric to the Euclidean sphere $\mathbb{S}^n$ if and only if the equality holds. 

If $m$ is an integer with $0\leq m\leq n$ and $\alpha$ is a scalar, then the $m$-th invariant of $\nabla^2 f$, is denoted by $S_m(f)$ and is defined by the condition (see, \cite{Robert1974} and \cite{Robert1977}, p. 461)
$$det(I+\alpha \nabla^2 f)=S_0(f)+\alpha S_1(f)+\cdots +\alpha^m S_m(f).$$

We recall that Hamilton conjectured in \cite{hamilton2, hamilton} that a compact gradient Ricci soliton  with positive curvature operator is an Einstein manifold (a trivial Ricci soliton),  which is settled in \cite{Bohm2008}. The next important question is to find the conditions, under which a compact gradient Ricci soliton becomes an Einstein manifold. Cao \cite{Cao2006} proved that every steady or expanding compact Ricci soliton is trivial. It is known that compact shrinking Ricci solitons are not necessarily trivial. In this context Chen and Deshmukh \cite{DC2014} studied compact shrinking Ricci soliton and obtained some necessary and sufficient conditions to become the solitons trivial. Now, generally a natural question arises,``Is it possible to impose some other conditions on compact shrinking Ricci soliton so that it become trivial ?". We have given an affirmative answer to this question in the following theorem:

\begin{thm}\label{Th6}
If $(M, g,  f)$ is a compact gradient shrinking Ricci soliton with constant $S_2(f)$, then the soliton is trivial, i.e., $M$ is Einstein.
\end{thm}
In the following theorem, we have given an application of the Bochner formula for gradient Ricci solitons and obtained a triviality result:

\begin{thm}\label{Th5} 
If $(M, g,  f)$ is a compact gradient Ricci soliton with constant $S_2(f)$, then there does not exist any non-trivial gradient Ricci soliton.
\end{thm}

If a gradient Ricci soliton is isometric to a quotient of $N\times\mathbb{R}^k$, it is said to be rigid (for more details about rigidity, see \cite{petersen1} and the references therein), where $N$ is an Einstein manifold. Petersen and Wylie \cite{petersen1}, studied gradient Ricci soliton and determined certain conditions, under which a gradient Ricci soliton becomes rigid. It is known that for compact Riemannian manifolds every steady or expanding Ricci soliton is rigid \cite{ivey}. Moreover, Hamilton  \cite{hamilton2} and Ivey \cite{ivey} proved that all shrinking compact Ricci solitons are rigid for the dimensions 2 and 3 respectively. Eminenti et al. \cite{Eminenti} proved that any $n$-dimensional compact shrinking Ricci soliton with constant scalar curvature is rigid. In \cite{petersen1} Petersen and Wylie obtained a little more general result and proved the following:
\begin{thm}[\cite{petersen1}]\label{Th8}
A compact gradient shrinking Ricci soliton is rigid if
\begin{equation*}
\int_{M} Ric( \nabla f, \nabla f )\leq 0.
\end{equation*}
\end{thm}
In the following theorem, we have proved a rigidity result for compact gradient shrinking Ricci soliton with some condition different from Petersen and Wylie \cite{petersen1}. In particular, we have proved the following theorem: 
\begin{thm}\label{Th10}
A compact gradient shrinking Ricci soliton $(M, g,  f)$ with constant $S_2(f)$ is rigid. 
\end{thm}

In the following three theorems, we have estimated upper bound for the conjugate radius of a Ricci soliton:
\begin{thm}\label{Th4}
Let $(M,g,f)$ be a non-compact gradient shrinking Ricci soliton with $Ric\geq 0$. If $f$ is superharmonic, then the conjugate radius of $M$ is given by
\begin{equation*}
r_{conj}(M)\leq \pi \sqrt{\frac{n-1}{\lambda}}.
\end{equation*}
\end{thm}
\begin{thm}\label{Th2}
Let $(M,g,X)$ be an $n(\geq 3)$-dimensional Ricci soliton with the potential vector field $X$ as concircular. Then the conjugate radius of $M$ is given by
\begin{equation*}
r_{conj}(M)\leq \pi\sqrt{\frac{n-1}{\lambda-\phi}},
\end{equation*}
where $(\lambda-\phi)$ is positive.
\end{thm}
\begin{thm}\label{Th3}
Let $(M,g)$ be a complete Riemannian manifold satisfying $Ric + \nabla^2f \geq\lambda g$. If the acceleration of $f$ along a geodesic ray $\sigma (t)$ is less than or equal to some real number k, then the conjugate radius of $M$ is given by
\begin{equation*}
r_{conj}(M)\leq \pi\sqrt{\frac{n-1}{\lambda-k}},
\end{equation*}
where $(\lambda-k)$ is positive.
\end{thm}

\begin{thm}\label{Th1}
Let $(M,g,f)$ be a non-compact gradient Ricci soliton with a pole at $p$ and $Ric\geq 0$. If the potential function $f$ is at most linear along a geodesic ray $\sigma(t)$, with $\sigma(0)=p$, then the soliton is non-shrinking.
\end{thm}

\section{Proof of the results}
In this section, we have presented the proof for the results of the paper.
First we have stated the following lemmas, which will be used in the sequel:
\begin{lem}[\cite{Robert1977}]\label{Robert1}
Let $(M,g)$ be a compact Riemannian manifold. Then the following holds:
\begin{equation*}
\int_{M} 2 S_2(f)=\int_{M}Ric(\nabla f, \nabla f).
\end{equation*}
\end{lem}
\begin{lem}[\cite{Robert1977}]\label{Robert2}
Let $(M,g)$ be a compact Riemannian manifold. If $f$ is a smooth function on $M$, then $S_2(f)$ can't be constant unless it vanishes.           
\end{lem}
The following relations hold for gradient Ricci solitons:
\begin{lem}[\cite{Cao2006, Zhang2009}]
Let $(M, g,  f)$ be a gradient Ricci soliton. Then the following holds:
\begin{equation}\label{eq9}
R+\Delta f=\lambda n,
\end{equation}
\begin{equation}\label{eq10}
\langle\nabla f, \nabla R\rangle=2 Ric(\nabla f,\nabla f).
\end{equation}
\end{lem}
\begin{proof}[\bf Proof of Theorem \ref{Th6}]
As $S_2(f)$ is a constant, by Lemma \ref{Robert2}, we infer that $S_2(f)$ vanishes. Therefore, by Lemma \ref{Robert1}, we have
\begin{equation}\label{eq15}
\int_{M}Ric(\nabla f, \nabla f)=0.
\end{equation}
Using (\ref{eq10}) in (\ref{eq15}), we obtain
\begin{equation*}
\int_{M}\langle\nabla f, \nabla R\rangle=0.
\end{equation*}
Now, by divergence theorem, we get
\begin{equation}\label{eq16}
\int_{M}R \Delta f=0.
\end{equation}
Again, as $\lambda$ is a constant, we have
\begin{equation}\label{eq17}
\int_{M} n\lambda \Delta f=0.
\end{equation}
Thus, combining (\ref{eq16}) and (\ref{eq17}), we obtain
\begin{equation}\label{eq18}
\int_{M} (n\lambda-R) \Delta f=0.
\end{equation}
The above equation (\ref{eq18}) together with equation (\ref{eq9}) yields
\begin{equation*}
\int_{M}(\Delta f)^2=0.
\end{equation*}
It follows that $\Delta f=0$. Therefore, $f$ is harmonic in a compact Riemannian manifold, hence $f$ is constant. This completes the proof. 
\end{proof}

\begin{proof}[\bf Proof of Theorem \ref{Th5}]
From (\ref{eq9}) and Bochner formula, we get
$$\frac{1}{2}\Delta|\nabla f|^2=|\nabla^2 f|^2-\langle\nabla f, \nabla R\rangle+Ric(\nabla f, \nabla f).$$
Now, using equation (\ref{eq10}), we obtain
\begin{eqnarray}\label{eq13}
\nonumber \frac{1}{2}\Delta|\nabla f|^2 &=& |\nabla^2 f|^2-2Ric(\nabla f, \nabla f)+Ric(\nabla f, \nabla f)\\
&=& |\nabla^2 f|^2- Ric(\nabla f,\nabla f).
\end{eqnarray}
Again, equation (\ref{eq15}) and (\ref{eq13}) together yield
\begin{equation*}
\int_{M} |\nabla ^2 f|^2=0.
\end{equation*}
It follows that $\nabla^2 f$ vanishes. As $M$ is compact, hence $f$ is constant. This completes the proof of the theorem. 
\end{proof}

\begin{proof}[\bf Proof of Theorem \ref{Th10}]
As $S_2(f)$ is a constant, by Lemma \ref{Robert2}, we get that $S_2(f)$ vanishes. Therefore, by Lemma \ref{Robert1}, we have
\begin{equation}\label{eq20}
\int_{M}Ric(\nabla f, \nabla f)=0.
\end{equation}
The equation (\ref{eq20}) together with Theorem \ref{Th8}, concludes the result.  
\end{proof}

Setting $l=r_{conj}(M)$, we consider any geodesic ball $B(x,l)$, $x\in M$, and any $q\in\partial B(x,l)$, there exists an arc length parameter $\sigma : [0,l] \rightarrow M$, which is the shortest geodesic connecting $x$ and $q$. Then, we have stated the following lemmas, which will be used in the sequel:
\begin{lem}[\cite{Zhu2022}] \label{Zhup1.4}
Let $(M,g)$ be a complete non-compact Riemannian manifold. Suppose that $B=B(p,r)\subset M$ is a geodesic ball with center at $p$ and radius $r>0$ and $B_{-l}=B(p,r-l)$. If $Ric\geq 0$, then
\begin{equation*}
\int_{B(p,r-l)}{ R }\leq n(n-1)\left(\frac{\pi}{l}\right)^2 Vol(B(p,r)).
\end{equation*}
\end{lem}
\begin{lem}[\cite{Zhu2022}] \label{Zhul2.7}
Let $(M,g)$ be a complete non-compact Riemannian manifold with a pole at $p$ and $Ric\geq 0$. Then for any geodesic ray $\sigma(t)$ with $\sigma(0)=p,$ the following holds
\begin{equation*}
\frac{1}{r}\int_{0}^{r}{t^2 Ric(\sigma'(t),\sigma'(t))}dt\leq {n-1}.
\end{equation*}
\end{lem}
\begin{proof}[\bf Proof of Theorem \ref{Th4}]
Tracing the gradient Ricci soliton equation (\ref{Grad Ricci soliton}), we have
\begin{equation}\label{eq7}
R+\Delta f=\lambda n.
\end{equation}
Since, $Ric\geq 0$, hence from (\ref{eq7}) and Lemma \ref{Zhup1.4}, we obtain
\begin{equation*}
\int_{B(p,r-l)}(\lambda n-\Delta f)\leq n(n-1)\left(\frac{\pi}{l}\right)^2 Vol(B(p,r)).
\end{equation*}
As $f$ is superharmonic, $\Delta f\leq 0$ , thus
\begin{equation}\label{eq8}
(\lambda n) Vol(B(p,r-l))\leq n(n-1)\left(\frac{\pi}{l}\right)^2 Vol(B(p,r)).
\end{equation}
By volume comparison theorem, we have (see, \cite{Zhu2022})
$$\lim\limits_{r\rightarrow \infty}\frac{Vol(B(p,r-l))}{Vol(B(p,r))}=1.$$
Therefore, from (\ref{eq8}), we get
\begin{equation*}
\lambda\leq \frac{(n-1)\pi^2}{l^2}.
\end{equation*}
This implies
\begin{equation*}
l\leq \pi \sqrt{\frac{n-1}{\lambda}}.
\end{equation*}

\end{proof}

\begin{proof}[\bf Proof of Theorem \ref{Th2}] We have (see \cite{Zhu2022}, p. 239)
\begin{equation*}
\int_{0}^{l}{sin^2\left(\frac{\pi}{l}t\right) Ric(\sigma'(t),\sigma'(t))}dt\leq \frac{(n-1)\pi^2}{2l}.
\end{equation*}
Using Ricci soliton equation (\ref{Ricci soliton}), we get
\begin{equation}\label{eq1}
\int_{0}^{l}{sin^2\left(\frac{\pi}{l}t\right) \{\lambda-\frac{1}{2}(\mathcal{L}_Xg)(\sigma'(t),\sigma'(t))\}}dt\leq \frac{(n-1)\pi^2}{2l}.
\end{equation}
For all $Y \in \chi(M)$, $\nabla_Y X=\phi Y$, as $X$ is concircular, where $\phi$ is a smooth function on $M$. Therefore
\begin{equation}\label{eq2}
\frac{1}{2}(\mathcal{L}_Xg) (\sigma'(t),\sigma'(t))=\phi g(\sigma'(t),\sigma'(t)).
\end{equation}
By (\ref{eq1}) and (\ref{eq2}), we get
\begin{equation}\label{eq3}
\int_{0}^{l}{sin^2\left(\frac{\pi}{l}t\right) (\lambda-\phi)}dt\leq \frac{(n-1)\pi^2}{2l}.
\end{equation}
Again, combining Ricci soliton equation (\ref{Ricci soliton}) and equation (\ref{eq2}), we have
\begin{equation*}
 Ric(\sigma'(t),\sigma'(t))=(\lambda-\phi )g(\sigma'(t),\sigma'(t)).
\end{equation*}
Since $n\geq 3$, the last relation implies that $M$ is Einstein (see \cite{chen2015}, p. 1543). Hence, $(\lambda-\phi)$ is a constant.
Therefore, from equation (\ref{eq3}), we get
\begin{equation*}
\frac{l}{2}(\lambda-\phi)\leq \frac{(n-1)\pi^2}{2l}.
\end{equation*}
This implies
\begin{equation*}
l\leq \pi\sqrt{\frac{n-1}{\lambda-\phi}}.
\end{equation*}
\end{proof}

\begin{proof}[\bf Proof of Theorem \ref{Th3}]
We have (see \cite{Zhu2022}, p. 239)
\begin{equation}\label{eq4}
\int_{0}^{l}{sin^2\left(\frac{\pi}{l}t\right) Ric(\sigma'(t),\sigma'(t))}dt\leq \frac{(n-1)\pi^2}{2l}.
\end{equation}
Using $Ric + \nabla^2f \geq\lambda g$ in (\ref{eq4}), we get
\begin{equation*}
\int_{0}^{l}{sin^2\left(\frac{\pi}{l}t\right) \{\lambda-\nabla^2 f(\sigma'(t),\sigma'(t))\}}dt\leq \frac{(n-1)\pi^2}{2l}.
\end{equation*}
As $f''(\sigma(t))\leq k$ , we have
\begin{equation}\label{eq6}
\int_{0}^{l}{sin^2\left(\frac{\pi}{l}t\right) (\lambda-k)}dt\leq \frac{(n-1)\pi^2}{2l}.
\end{equation}
Hence, equation (\ref{eq6}) entails
\begin{equation*}
\frac{l}{2}(\lambda-k)\leq \frac{(n-1)\pi^2}{2l}.
\end{equation*}
This implies
\begin{equation*}
l\leq \pi\sqrt{\frac{n-1}{\lambda-k}}.
\end{equation*}
\end{proof}

\begin{proof}[\bf Proof of Theorem \ref{Th1}]
For any $r>0$, the Lemma \ref{Zhul2.7} together with gradient Ricci soliton equation (\ref{Grad Ricci soliton}), yields
\begin{equation*}
\frac{1}{r}\int_{0}^{r}t^2\left\{\lambda-\nabla^2 f (\sigma',\sigma')\right\}dt\leq {n-1}.
\end{equation*}
Thus
\begin{equation}\label{eq23}
\frac{1}{r}\left\{\frac{\lambda r^3}{3}-\int_{0}^{r} t^2 f'' (\sigma(t))dt\right\}\leq {n-1}.
\end{equation}
Now, a simple calculation gives
\begin{equation}\label{eq24}
\int_{0}^{r} t^2f'' (\sigma(t))dt=r^2f' (\sigma(r))-2r f(\sigma(r))+2\int_{0}^{r} f(\sigma(t))dt.
\end{equation}
In view of (\ref{eq23}) and (\ref{eq24}), we obtain
\begin{equation*}
\frac{1}{r}\left\{\frac{\lambda r^3}{3}-r^2f' (\sigma(r))+2r f(\sigma(r))-2\int_{0}^{r} f(\sigma(t))\right\}dt\leq {n-1}.
\end{equation*}
This implies
\begin{equation*}
\frac{\lambda}{3}-\frac{1}{r}f' (\sigma(r))+\frac{2}{r^2} f(\sigma(r))-\frac{2}{r^3}\int_{0}^{r} f(\sigma(t))dt\leq \frac{n-1}{r^2}.
\end{equation*}
As $f$ is at most linear along $\sigma$, taking limit as $r \rightarrow \infty $, we obtain
$$\lambda\leq 0.$$
It follows that the Ricci soliton is non-shrinking. 
\end{proof}

\section*{Acknowledgment}
 The second author gratefully acknowledges to the CSIR(File No.:09/025(0282)/2019-EMR-I), Govt. of India for the award of Senior Research Fellow.

 \section{Data Availability Statement}
 Our manuscript has no associated data.
 \section{Declarations: Conflict of Interest}
 The authors declare that they have no conflict of interest.

\end{document}